\documentclass[a4paper]{article}

\title{Some Definability Results in Abstract Kummer Theory}

\usepackage{amssymb}
\usepackage{amsmath}
\usepackage{amsthm}
\usepackage{amsfonts}
\usepackage{url}
\usepackage{enumerate}
\usepackage{showlabels}
\theoremstyle{plain}
\newtheorem{theorem}{Theorem}[section]
\newtheorem{lemma}[theorem]{Lemma}
\newtheorem{proposition}[theorem]{Proposition}

\newtheorem{fact}[theorem]{Fact}
\newtheorem*{fact*}{Fact}
\newtheorem{corollary}[theorem]{Corollary}
\newtheorem{claim}{Claim}
\newtheorem*{claim*}{Claim}

\theoremstyle{definition}
\newtheorem{definition}[theorem]{Definition}

\newtheorem*{definition*}{Definition}

\newtheorem*{notation*}{Notation}

\theoremstyle{remark}
\newtheorem{remark}[theorem]{Remark}
\newtheorem*{remark*}{Remark}
\newtheorem{examples}[theorem]{Examples}
\newtheorem*{examples*}{Examples}
\newtheorem{question}[theorem]{Question}
\newtheorem*{question*}{Question}

\newtheorem*{note*}{Note}
\newtheorem{example}{Example}
\newtheorem*{example*}{Example}

\renewcommand{\phi}{\varphi}

\newcommand{\elres}{\succcurlyeq}

\newcommand{\isom}{\cong}
\newcommand{\theorystyle}[1]{\operatorname{#1}}
\newcommand{\ACF}{\theorystyle{ACF}}
\newcommand{\ACFA}{\theorystyle{ACFA}}
\newcommand{\DCF}{\theorystyle{DCF}}

\newcommand{\Th}{\operatorname{Th}}
\newcommand{\pr}{\operatorname{pr}}
\newcommand{\tp}{\operatorname{tp}}
\newcommand{\stp}{\operatorname{stp}}

\newcommand{\RM}{\operatorname{RM}}
\newcommand{\DM}{\operatorname{DM}}
\newcommand{\U}{\operatorname{U}}
\newcommand{\RDM}{\operatorname{RDM}}
\newcommand{\acl}{\operatorname{acl}}
\newcommand{\id}{\operatorname{id}}
\newcommand{\dcl}{\operatorname{dcl}}
\newcommand{\acleq}{\operatorname{acl}^{\operatorname{eq}}}

\newcommand{\N}{\mathbb{N}}
\newcommand{\Z}{\mathbb{Z}}

\newcommand{\Q}{\mathbb{Q}}

\newcommand{\tuple}[1]{\overline{#1}}
\newcommand{\btup}{\tuple{b}}
\newcommand{\dtup}{\tuple{\delta}}
\newcommand{\atup}{\tuple{a}}

\newcommand{\xtup}{\tuple{x}}

\newcommand{\Gal}{\operatorname{Gal}}
\newcommand{\locus}{\operatorname{locus}}
\newcommand{\defnstyle}[1]{\emph{#1}}
\newcommand{\maps}{\rightarrow}
\newcommand{\invlim}{\mathop{\varprojlim}\limits}

\newcommand{\Zar}{{\mathrm{Zar}}}
\newcommand{\G}{\mathbb{G}}

\newcommand{\Gt}{\tilde{G}}
\newcommand{\thetat}{\tilde{\theta}}

\renewcommand{\L}{\mathcal{L}}

\newcommand{\res}{\upharpoonright}

\newcommand{\COMMENTSAI}[1]{}  

\author{Martin Bays\thanks{The first author was supported, and the third author
was partially funded, by MODIG (Projet ANR-09-BLAN-0047)}, Misha
Gavrilovich\thanks{The second author was supported by the FWF through project
AM1202.} and Martin Hils}

\begin{document}

\maketitle

\begin{abstract}
  Let $S$ be a semiabelian variety over an algebraically closed field, and let 
  $X$ be an irreducible subvariety not contained in a translate of a proper 
  algebraic subgroup of $S$. We show that the number of irreducible 
  components of $[n]^{-1}(X)$ is bounded uniformly in $n$, and moreover 
  that the bound is uniform in families $X_t$.

  We prove this by Galois-theoretic methods. This proof allows for a purely
  model theoretic formulation, and applies in the more general context of
  divisible abelian groups of finite Morley rank. In this latter context, we
  deduce a definability result under the assumption of the Definable
  Multiplicity Property (DMP). We give sufficient conditions for finite Morley
  rank groups to have the DMP, and hence give examples where our definability
  result holds.
\end{abstract}

\bibliographystyle{alpha}

\section{Introduction}
Let $S$ be a semiabelian variety over an algebraically closed field of
characteristic $p\geq 0$, and let $X$ be an irreducible subvariety. We say that
$X$ is \defnstyle{Kummer-generic} (in $S$) if $[n]^{-1}(X)$ is irreducible for
all $n$. Here, $[n]:S\rightarrow S$ denotes the multiplication-by-$n$ map.

It is easy to see that, since $S$ has Zariski-dense torsion, a necessary
requirement for $X$ to be Kummer-generic is that it is not contained in any
translate of a proper algebraic subgroup of $S$ --- we call such $X$
\defnstyle{free}.

If $X$ is free, it does not follow that $X$ itself is Kummer-generic. For
example, $X := \{(x,y) \;|\; y = (1+x)^2 \} \subseteq \G_m^2$ is irreducible
and free, but $[2]^{-1}(X) = \{(x,y) \;|\; y^2 = (1+x^2)^2\}$ is not
irreducible, since $(y+1+x^2)(y-(1+x^2)) = y^2 - (1+x^2)^2$.

However, we prove:

\begin{theorem}\label{thm:ag1}
  Suppose $X\subseteq S$ is free. Then for some $n$, any irreducible component
  of $[n]^{-1}(X)$ is Kummer-generic in $S$.
\end{theorem}

\begin{theorem}\label{thm:ag2}
  Suppose $S\maps T$ is a parametrised family of semiabelian varieties and
  $X\subseteq S$ is a family of subvarieties. Then
  $\{t \;|\; \textrm{$X_t$ is Kummer-generic in $S_t$}\}$ is a constructible
  set.
\end{theorem}

Here, a family $S\maps T$ of semiabelian varieties is assumed to have uniform
group structure, i.e. a morphism $+ : S\times S \maps S$ restricting to the
group morphisms $+_t : S_t\times S_t \maps S_t$ of the semiabelian varieties
$S_t$.

In fact, we prove Theorem~\ref{thm:ag2} by proving a uniform version of
Theorem~\ref{thm:ag1}, Proposition~\ref{prp:main} below.

Moreover, we prove generalisations of Theorems \ref{thm:ag1} and \ref{thm:ag2}
(Theorems \ref{thm:fmr1} and \ref{thm:fmr2} respectively) in the context of
divisible abelian groups of finite Morley rank.
Theorems \ref{thm:ag1} and \ref{thm:ag2} may be easily derived from these more
abstract counterparts, in fact with the slightly weaker assumption that the
algebraic group be commutative and divisible. Nonetheless, we think it worth
keeping the (purely algebraic) proofs of Theorems \ref{thm:ag1} and
\ref{thm:ag2}.

Versions of Theorem~\ref{thm:ag1} have appeared previously in work on the
model theory of universal covers of commutative algebraic groups, under the
guise of ``the $n=1$ case of the Thumbtack Lemma''. Zilber (\cite{ZCovers})
proves Theorem~\ref{thm:ag1} in the case that $S=\G_m^n$ is an algebraic torus
in characteristic 0, by consideration of divisors on a projective normal model
of the function field of $X$; in \cite{BZCovers} it was noted that this
argument goes through in positive characteristic. The second author
(\cite{GavThesis,GavCoversAZ}) proved Theorem~\ref{thm:ag1} for Abelian
varieties in characteristic 0 by complex analytic and homotopic methods -- in
fact, the proof there proceeds by considering the fundamental group action on
the universal covering space, and does not explicitly use the group structure
on the abelian variety. The first author (\cite{BaysThesis}) gave an
alternative proof appealing to Lang-N\'{e}ron's function field Mordell-Weil
theorem.

Meanwhile, versions of Theorem~\ref{thm:ag2} have arisen in the study of
``green fields'', fields expanded by a predicate for a generic multiplicative
subgroup. In considering the theory of generic automorphisms of green fields,
the third author (\cite{HilsGenAutGreen}) found that he needed
Theorem~\ref{thm:ag2} for tori, and proved it by a finer consideration of
divisors in the style of Zilber's proof of Theorem~\ref{thm:ag1} for tori.
Subsequently, it was observed by Roche that Theorem~\ref{thm:ag2} for tori is
a necessary part of the ``collapse" of the green fields onto so-called
\emph{bad fields }by Baudisch, Martin Pizarro, Wagner and the third author
(\cite{BaHiMaWa09}).

The present work has corresponding applications. Indeed, in his thesis
\cite{RocheThesis}, Roche considers so-called \emph{octarine fields}, certain
expansions of abelian varieties by a predicate for a non-algebraic subgroup, a
context which is similar to bad fields. In order to to be able to perform the
``collapse" in this context, definability of Kummer-genericity (for abelian
varieties in characteristic 0) is needed.

A case of Theorem~\ref{thm:ag1}, for hypersurfaces in tori in characteristic
0, appears in work of Ritt \cite[Section 8]{RittFact} in the context of
factorising exponential sums.

\smallskip

The key idea of our proof is due to Ofer Gabber. In fact, he provided a proof
(sketch) of Theorems~\ref{thm:ag1} and \ref{thm:ag2}, going via (\'{e}tale)
fundamental groups. We worked out the details and simplified his proof. In
particular, we were able to extract the `Galois theoretic' essence of the
arguments and give proofs that do not use algebraic geometry in an essential
way, replacing the use of the \'etale fundamental group of the variety by the
absolute Galois group of the function field of the variety. Galois theory is
available in a first-order structure (\cite{Poi83}), and the
statements and proofs transfer to this abstract model theoretic setting.
Analysing the proofs, one sees that, in model theoretic terms, we use only
that the complete theories $\ACF_p$ of algebraically closed fields have finite
Morley rank and the Definable Multiplicity Property (DMP). Analysing this, we
find that we can prove analogous statements to Theorems~\ref{thm:ag1} and
\ref{thm:ag2} under more general conditions (partially even for type-definable
groups of finite relative Morley rank); we do this in Section \ref{S:fMR},
transferring the relevant portions of Kummer theory to the more abstract
setting.

First, in Section \ref{S:TA-DMP}, we establish a criterion for a group $G$ of
finite Morley rank (maybe with additional structure) to have the DMP which
turns out to be very useful in practice. We show that $G$ has the DMP if and
only if the generic automorphism is axiomatisable in $G$. Indeed, this
equivalence holds for non-multidimensional theories with all dimensions
strongly minimal. This generalises the case of strongly minimal theories which
was proved by Hasson and Hrushovski \cite{HaHr07}.

These results apply to any group of finite Morley rank definable in the theory
$\DCF_0$ of differentially closed fields of characteristic 0
(such groups are (\cite[Proposition~2.4]{KowalskiPillayDGrps}) definably isomorphic
to groups of the form $(G,s)^{\#} = \{ x\in G \;|\; s(x)=\delta x\}$ for
$(G,s)$ an algebraic D-group), as well as to groups definable in compact
complex manifolds (which are known (\cite{PillayScanlonMeroGrps},
\cite{ScanlonNSMeroGrps}) to be definable extensions of definably compact
groups by linear algebraic groups), and we thus obtain Theorems~\ref{thm:ag1}
and \ref{thm:ag2} in these contexts. In the differential case, we use that the
generic automorphism is axiomatisable in $\DCF_0$ (a result of Hrushovski, see
\cite{Bus07}); in the compact complex analytic case, we obtain the
axiomatisability of the generic automorphism using a result of Radin
\cite{Rad04} which asserts that irreducibility is definable in families.

Finally, we would like to mention an instance of the type-definable case. Let
$A$ be an abelian variety which is defined over a non-perfect separably closed
field $K$ of characteristic $p>0$ (of finite degree of imperfection), and let
$A^{\#}$ be the biggest $p$-divisible subgroup of $A(K)$ (equivalently,
$A^{\#}$ may be defined as the smallest type-definable subgroup which is
Zariski dense in $A$). Then Theorem~\ref{thm:ag1} holds in $A^{\#}$ (for
relatively definable subsets), since $A^{\#}$ has relative finite Morley rank
(\cite{BeBoPiSSharp}).

We are naturally very grateful to Ofer Gabber for sharing his insights with
us. This interaction, and collaboration between the authors, occurred during a
visit by the second author to the IHES, and we would like to thank them for
their hospitality and providing an environment conducive to such interaction.

We would also like to thank Zo\'{e} Chatzidakis, Philipp Habegger, Rahim Moosa,
and Giuseppina Terzo for helpful remarks.

\section{Kummer theory}\label{S:Kummer}

We first reformulate Kummer genericity as a condition on the image of the
absolute Galois group of a function field on the product of the Tate modules.

Let $S$ be a semiabelian variety over an algebraically closed field. By
standard results, $S$ is divisible.

Let $l\in\N$. The l-torsion $S[l]$ of $S$ is finite by l-divisibility of $S$
and dimension concerns. So $S[l]$ is a finite abelian group of exponent $l$,
hence isomorphic to some $(\Z/l\Z)^{k_l}$ if $l$ is a prime number.

Let $T = T_S$ be the inverse limit of the torsion
    \[T = \invlim_n S[n]\]
with respect to the multiplication-by-m maps
    \[[m] : S[mn] \maps S[n] .\]
Splitting into primary components, we have an isomorphism of profinite groups:
    \[T \isom \Pi_l \Z_l^{k_l}\]

(Although we will not need this, in fact $k_l$ depends only on whether $l=p$:
if $0\maps T \maps S \maps A \maps 0$ is exact with $A$ an abelian variety and
$T$ a torus, then $k_l = 2\dim(A) + \dim(T)$ if $l\neq p$, and $0\leq k_p \leq
\dim(A)$. It is worth noting that we need no assumption on $A$ being ordinary
in order to prove our results.)

Let $X\subseteq S$ be an irreducible subvariety, let $K$ be algebraically
closed such that $X$ and $S$ are defined over $K$, and let $b\in X$ be a
generic point over $K$.

Let $G := \Gal(K(b)^{alg}/K(b))$.

The Kummer pairing provides a continuous homomorphism
\begin{align*}
    \theta :\,\,& G \maps T \\
	  & \sigma \mapsto ( \sigma(b_n) - b_n )_n
\end{align*}
where the $b_n$ are arbitrary such that  $m b_{nm} = b_n$  and  $b_1=b$; this
is a well-defined homomorphism, since the torsion of $S$ is algebraic and hence
contained in $S(K)$.

Let $Z = Z_X \leq T$ be the image of $\theta$. This does not depend on the
choices of $K$ and $b$. It follows from continuity of $\theta$ that $Z$ is a
closed subgroup of $T$.

Say $X$ is \emph{$n$-Kummer-generic} in $S$ if and only if $[n]^{-1}(X)$ is
irreducible.

\begin{lemma}\label{lem:KgZ}
  $X$ is $n$-Kummer-generic in $S$ iff $Z+nT=T$; hence $X$ is
  Kummer-generic in $S$ iff $Z = T$.
\end{lemma}
\begin{proof}
  The map $[n]$ is closed - indeed, since $[n]$ generically has finite fibres,
  there exists an open $U\subseteq S$ such that $[n]\restriction_{[n]^{-1}(U)}$
  is finite in the sense of algebraic geometry, and hence closed; covering $S$
  with translates of $U$, we see that $[n]$ is closed.

  It follows that some irreducible component $Y \subseteq [n]^{-1}X$ is such
  that $[n]Y = X$. If $Y'$ is another irreducible component, by considering
  its generic we see that $Y' \subseteq Y+\zeta\subseteq [n]^{-1}X$ for some
  $\zeta\in S[n]$; hence $Y'=Y+\zeta$. So $\{ Y + \zeta | \zeta \in S[n] \}$
  is the irreducible decomposition of $[n]^{-1}X$.

  Let $b$ be generic in $X$ over $K$ and let $b_n\in Y$ be such that $nb_n=b$,
  so $b_n$ is generic in $Y$ over $K$. Then
  \begin{align*}
    &\,\,\,\, Z+nT=T\\
    \Leftrightarrow &\,\,\,\, b_n \text{ is conjugate over }K
    \text{ to }b_n+\zeta\text{ for all }\zeta\in S[n]\\
    \Leftrightarrow &\,\,\,\, Y = Y+\zeta\text{ for all }\zeta\in S[n]\\
    \Leftrightarrow &\,\,\,\, [n]^{-1}X\text{ is irreducible.}
  \end{align*}

  The last equivalence is by our description of the decomposition.\end{proof}

\begin{lemma}\label{lem:Kgprimes}
  $X$ is Kummer-generic in $S$ iff $X$ is $l$-Kummer-generic in $S$
  for all primes $l$.
\end{lemma}
\begin{proof}
  By Lemma~\ref{lem:KgZ}, $X$ is $l$-Kummer-generic iff $Z+lT=T$;
  taking $l$-primary components $Z_l$ of $Z$ and $T_l$ of $T$, this is
  equivalent to $Z_l+lT_l=T_l$. But $T_l \isom \Z_l^{k_l}$, and $Z_l$ is a
  closed subgroup and hence a $\Z_l$-submodule, so $Z_l+lT_l=T_l$ iff
  $Z_l=T_l$ (to see this: note $Z_l+lT_l=T_l$ implies that, with respect to
  the isomorphism $T_l \isom \Z_l^{k_l}$, the submodule $Z_l$ contains the
  columns of a matrix equal modulo $l\Z_l$ to the identity matrix; but this
  matrix has determinant in $\Z_l^* = \Z_l\setminus l\Z_l$, hence is
  invertible in $\operatorname{Mat}_{k_l}(\Z_l)$). The lemma follows easily.
\end{proof}

In the next section, we prove:
\begin{proposition}\label{prp:main}
  Let $X_t \subseteq S_t$ be a family of free irreducible subvarieties of a
  family of semiabelian varieties. Then $Z_{X_t}$ is of finite index in
  $T_{S_t}$, and moreover this index is bounded.
\end{proposition}

\begin{proof}[Proof of Theorem~\ref{thm:ag1} from Proposition~\ref{prp:main}]
  By Proposition~\ref{prp:main} with a constant family, $Z$ is of finite index
  in $T$, say $m$. So $m$ annihilates $T/Z$, hence $mT \leq Z$. If $Y$ is an
  irreducible component of $[m]^{-1}X$, it follows that $Z_Y=T$, i.e. that $Y$
  is Kummer-generic in $S$.
\end{proof}

\begin{proof}[Proof of Theorem~\ref{thm:ag2} from Proposition~\ref{prp:main}]
  It is clear that Kummer-genericity implies irreducibility. Irreducibility is
  a constructible condition - the set of $t$ for which $X_t$ is irreducible
  is constructible - so we may assume that every $X_t$ in our family is
  irreducible.

  Since the torsion of $S$ is Zariski-dense, Kummer-genericity implies
  freeness. Freeness is also a constructible condition -- indeed, by Claim
  \ref{claim:dom} below, $X\subseteq S$ is free iff the summation map
  $\Sigma:X^{2d}\maps S$ is surjective, where $d=\dim(S)$. So we may assume
  that for all $t$, $X_t$ is free in $S_t$.

  Irreducibility is a constructible condition, hence so is
  $l$-Kummer-genericity. By Lemma~\ref{lem:Kgprimes} and
  Proposition~\ref{prp:main}, for Kummer-genericity we need check only
  finitely primes $l$; hence, Kummer-genericity is also a constructible
  condition.
\end{proof}

\section{Proof of Proposition~\ref{prp:main}}\label{S:Proof2.3}
We prove this for a fixed $X\subseteq S$, and then note that the proof gives
uniformity in families.

Let $d := \dim(S)$.

\begin{claim}   \label{claim:dom}
 The summation map
  \[\Sigma : X^d \maps S\]
  is dominant.
\end{claim}

\begin{proof}
  It suffices to show that if $k\leq d$ and the image of the $k$-ary summation
  map $X_k := \Sigma_1^k(X^k)$ has dimension $\dim(X_k)<d$, then $\dim(X_k+X)
  > \dim(X_k)$. Translating, we may assume $0\in X$. By irreducibility it is
  then enough to see that, taking Zariski closures, $(X_k+X)^{\Zar} \neq
  X_k^{\Zar}$; but else, we would have $(X_k+X_k)^{\Zar} = X_k^{\Zar}$, whence
  $X_k^{\Zar}$ is a (proper) subgroup, contradicting freeness of $X$.
\end{proof}

Let $K$ be an algebraically closed field over which $X$ and $S$ are defined.

Let $a\in S$ be generic over $K$.

Let $\Gt := \Gal(K(a)^{alg}/K(a))$.

Let $P$ be the set of (absolutely) irreducible components of $\Sigma^{-1}(a)$.

$\Gt$ acts naturally on $P$. The action is transitive, since $X^d$ is
irreducible and, by smoothness of $S$ and the dimension theorem, all
irreducible components of $\Sigma^{-1}(a)$ have full dimension $\dim(X^d) -
\dim(S)$.

Let $Z=Z_X\leq T_S=T$, defined as above.

\providecommand{\Stab}{\operatorname{Stab}}
\begin{claim}\label{clm:maingal}
  Let $V\in P$. The map $\thetat$ : $\Gt \maps T$
  \[\thetat : \sigma \mapsto ( \sigma(a_n) - a_n )_n\]
  (where $(a_n)$ is such that $m a_{nm} = a_n$ and $a_1=a$) induces a
  group epimorphism
  \[\Gt \maps T/Z\]

  whose kernel contains $\Stab_{\Gt}(V)$; hence the index of $Z$ in $T$ divides
  $|\Gt/\Stab_{\Gt}(V)| = |P|$, and in particular is finite.
\end{claim}
\begin{proof}
  Obviously, $S$ is Kummer-generic in itself, as $[n]^{-1}(S) = S$ is
  irreducible for all $n$. Thus $Z_S = T$, showing that the map
  $\tilde{\theta}$ is surjective.

  Now let $\btup$ be generic in $V$ over $K(a)$. Note that $\btup$ is
  generic in $X^d$ over $K$.

  Let $\sigma \in \Stab_{\Gt}(V) \leq \Gt$. Then $\sigma$ extends to $\tau \in
  \Gal(K(a,\btup)^{alg}/K(a,\btup))$. Choose $(\btup_n)_{n\in\N}$ satisfying
  $m \btup_{nm} = \btup_n$ and $\btup_1=\btup$. We then have
  \begin{align*}
    \thetat(\sigma)  &= ( \sigma(\Sigma\btup_n) - \Sigma\btup_n )_n \\
    &= ( \Sigma (\tau\btup_n - \btup_n) )_n \\
    &= ( \tau(\btup_n)_1 - (\btup_n)_1 )_n + ... +
    ( \tau(\btup_n)_d - (\btup_n)_d )_n
  \end{align*}
  But this is an element from $Z+\ldots+Z=Z$, and so $\thetat$ induces a map
  as stated.
\end{proof}


This claim proves the non-uniform part of Proposition~\ref{prp:main}. Since
the number of irreducible components of the generic fibre of $\Sigma$ is
bounded uniformly in families, we obtain Proposition~\ref{prp:main}.

\smallskip

In our initial proof, we used the following remark to reduce the statement
to the case of curves. This is no longer needed, and so we include it without
proof. (One may use arguments similar to \cite[Proof of Lemme 3.1]{Poi01}.)

\begin{remark}
  Let $X\subseteq S$ be a Kummer-generic variety of dimension at least 2. Let
  $H\subseteq S$ be a generic hyperplane (with respect to some embedding of
  $S$ into some $\mathbb{P}^n$). Then $X\cap H$ is Kummer-generic.
\end{remark}

We would also like to remark that the proof of Proposition~\ref{prp:main}
applies generally to endomorphisms of abelian algebraic groups with finite
kernel, in particular to the Artin-Schreier map
$\wp:\G_a^n\maps \G_a^n;\; \xtup\mapsto (x_i^p-x_i)_i$ in characteristic
$p\neq 0$, and its compositional iterates $\wp^{(m)}$. We therefore obtain an
analogue of Theorem~\ref{thm:ag1}:

\begin{remark}
  Say $X\subseteq \G_a^n$ is \defnstyle{Artin-Schreier-generic} iff
  $(\wp^{(m)})^{-1}X$ is irreducible for all $m$. So if $X\subseteq
  \G_a^n$ is free in the sense that the summation map $\Sigma : X^n\maps
  \G_a^n$ is dominant, then for some $m$ each irreducible component of
  $(\wp^{(m)})^{-1}X$ is Artin-Schreier-generic.
\end{remark}
%

\section{DMP in groups of finite Morley rank} \label{S:TA-DMP}
The aim of the rest of this paper is to give model theoretic generalisations
of Theorems~\ref{thm:ag1} and \ref{thm:ag2}. We use standard model theoretic
terminology and concepts; see e.g.\ \cite{Pil96} for these.

We show here a number of
results which will be useful when generalising our results for semiabelian
varieties to groups of finite Morley rank.

In this and the following section, we will often write $x,y,\ldots$ for
(finite) tuples of variables; similarly, $a,b,\ldots$ may denote
tuples of elements. The Morley rank of a (partial) type $\pi$ will be denoted
$\RM(\pi)$, its Morley degree by $\DM(\pi)$. Moreover, we write $\RDM(\pi)$
for the pair $\left(\RM(\pi),\DM(\pi)\right)$.

\begin{definition}
  We say that the theory $T$ is \emph{almost $\aleph_1$-categorical}
  if $T$ is non-multidimensional, with all dimensions strongly
  minimal, i.e.\ there is a fixed set of strongly minimal sets $\{D_i\mid\,
  i\in I\}$ (in $T^{eq}$) such that every non-algebraic type is non-orthogonal
  to one of the $D_i$'s.
\end{definition}

Note that in the previous definition we do not assume that the language is
countable.

Observe that $T$ is almost $\aleph_1$-categorical iff $T^{eq}$ is. Moreover,
in an almost $\aleph_1$-categorical theory, there are only finitely many 
non-orthogonality classes of strongly minimal types, so we may assume that
$I$ is finite. For background on almost $\aleph_1$-categorical theories, in
particular a proof of the following fact, we refer to \cite{PiPo02}.

\begin{fact}\label{fact:PillayPong}
  Let $T$ be almost $\aleph_1$-categorical.

  Then, in $T^{eq}$, Morley rank is finite, definable and equal to Lascar
  $\U$-rank.
\end{fact}

Let $T$ be a theory of finite Morley rank. Recall the following definitions:
\begin{itemize}
  \item Morley rank is \emph{definable} in $T$ if for every formula
    $\phi(x,y)$ and every $r\in\N$ there is a formula $\theta(y)$ such that
    $\RM(\phi(x,a))=r$ iff $\models\theta(a)$.
  \item $T$ has the \emph{definable multiplicity property} (DMP) if for every formula $\phi(x,y)$ and every $(r,d)\in\N^2$ there is a formula
    $\theta(y)$ such that $\RDM(\phi(x,a))=(r,d)$ iff $\models\theta(a)$.
\end{itemize}

To establish the DMP, by compactness, it is enough to show that whenever
$\RDM(\phi(x,a))=(r,d)$, there is $\theta(y)\in\tp(a)$ such that
$\RDM(\phi(x,a'))=(r,d)$ whenever $\models\theta(a')$.

\begin{remark}
  Almost $\aleph_1$-categorical theories were introduced in \cite{ Belegradek};
  though stated in terms of 
  extensions of models,
  his definition is quite close to the one we use. These theories were
  further studied by \cite{ErimbetovACT} who in particular proved the
  finiteness of Morley rank.
\end{remark}

\begin{fact}[Lascar \cite{Las85}]\label{fact:Lascar}
  Let $G$ be a group of finite Morley rank and $T=\mathrm{Th}(G)$. Then $T$ is
  almost $\aleph_1$-categorical. In particular, Morley rank is finite,
  definable and equal to Lascar $\U$-rank.
\end{fact}

\begin{definition}
  Let $p(x)$ be a stationary type. Then $p$ is said to be \emph{uniformly of Morley
  degree 1} if there is a formula $\phi(x,a)\in p$ and $\theta(y)\in\tp(a)$
  such that $p$ is the unique generic type in $\phi(x,a)$ and
  $\DM(\phi(x,a'))=1$ whenever $\models\theta(a')$.
\end{definition}

Note that if Morley rank is definable in $T$ and $p$ is uniformly of Morley
degree 1, we may witness this by formulas $\phi$ and $\theta$ such that
$\RM(\phi(x,a'))=\RM(p)$ whenever $\models\theta(a')$.

\smallskip

For $p,q\in S(B)$, we say that $p$ is a \emph{finite cover }of $q$ if there
are $a\models p$ and $b\models q$ such that $b\in\dcl(Ba)$ and $a\in\acl(Bb)$.
The proof of the following lemma is easy and left to the reader.

\begin{lemma}\label{lem:uniformMD1}
  Assume that Morley rank is finite and definable in $T$, and let $p$ be
  uniformly of Morley degree 1.
  \begin{enumerate}
    \item If $q$ is parallel to $p$, then $q$ is uniformly of Morley degree 1.
    \item If $p$ is a finite cover of $q$, then $q$ is uniformly of Morley
      degree 1.
    \item $T$ has the DMP iff every stationary type is uniformly of
      Morley degree 1.\qed
  \end{enumerate}
\end{lemma}

\begin{lemma}\label{lem:charDMPsm}
  Assume that Morley rank is finite, definable and equal to $\U$-rank in
  $T=T^{eq}$. Then $T$ has the DMP iff every strongly minimal type
  is uniformly of Morley degree 1.
\end{lemma}

\begin{proof}
  For $T$ strongly minimal, this is shown in \cite{Hru92}.

  The condition is clearly necessary. In order to show that it is sufficient,
  it is enough to show that any stationary type $p$ is uniformly of Morley
  degree 1, by Lemma \ref{lem:uniformMD1}(3). We prove this by induction on
  $\RM(p)=r$, the case $r=1$ being true by assumption (and $r=0$ being
  trivial).

  Assume the result is true for types of Morley rank $\leq r$ and let $p(x)\in
  S(B)$ be stationary, $\RM(p)=r+1$.

  Since $0<\U(p)<\omega$, there is a minimal type $p_1(x_1)$ (i.e.
  $\U(p_1)=1$) such that $p\not\perp p_1$ (see \cite[Lemma 2.5.1]{Pil96}).
  By Lemma \ref{lem:uniformMD1}(1), we may replace $p$ and $p_1$ by
  non-forking extensions and thus assume that $p,p_1\in S(M)$ for some model
  $M$ and $p\not\perp^{a} p_1$, whereby there are $a\models p$ and $a_1\models
  p_1$ such that $a_1\in\acl(Ma)$. By Lemma \ref{lem:uniformMD1}(2), we may
  replace $p$ by the finite cover $\tp(aa_1/M)$ and thus assume that $p_1$ is
  given by the restriction of $p(x)$ to a subtuple $x_1$ of the variables. Let
  $a_2$ be such that $a=a_1a_2$. By assumption, $\U=\RM$, so $p_1$ is strongly
  minimal. Let $q(x_2)=\tp(a_2/Ma_1)$. By $\omega$-stability,
  $\tp(a_2/M\tilde{a}_1)$ is stationary for some finite $\tilde{a}_1$ with
  $a_1\subseteq\tilde{a}_1\subseteq\acl(Ma_1)$. Enlarging $a_1$ if necessary,
  we may thus assume that $q$ is stationary.

  Moreover, the Lascar inequalities and $\RM=\U$ imply that $\RM(q)=r$.

  Using definability of Morley rank and the induction hypothesis, we may find
  an $\L$-formula $\phi(x_2,x_1,z)$, $b\in M$ and $\theta(z)\in\tp(b)$ such
  that

  \begin{itemize}
    \item[(i)] $p$ is the unique generic type in $\phi(x_2,x_1,b)$ over $M$;
    \item[(ii)] $\RM(\phi(x_2,x_1,b'))=r+1$ whenever $\models\theta(b')$;
    \item[(iii)] $\psi(x_1,b')=\exists x_2\phi(x_2,x_1,b')$ is strongly
      minimal whenever $\models\theta(b')$ (in particular, $p_1$ is the unique
      generic type in $\psi(x_1,b)$ over $M$);
    \item[(iv)] whenever $\phi(x_2,a_1',b')$ is consistent,
      $\RDM(\phi(x_2,a_1',b'))=(r,1)$.
  \end{itemize}

  It is routine to check that $\phi(x_2,x_1,b)$ and $\theta(z)$ witness that
  $p$ is uniformly of Morley degree 1.
\end{proof}

It follows in particular from the previous lemma that in any strongly minimal theory without the DMP, there 
is a strongly minimal type $p$ which is not uniformly of Morley degree 1. To illustrate this, let us 
recall the following example due to Hrushovski \cite{Hru92}.

\begin{example}
Let $V$ be a non trivial vector space over $\Q$ and $0\neq v_0\in V$. Let $D = V \times \{0, 1\}$  
equipped with the projection $\pi: D \rightarrow V$ and the function $f : D \rightarrow D$, $f(v,i)=(v+v_0,i)$. 
Then $T=\mathrm{Th}(D)$ is strongly minimal without the DMP. For any $b\in D$, the formula 
$\phi(x,y,b)$ given by $\pi(x)=\pi(y)+\pi(b)$ is of Morley rank 1, and it is strongly minimal iff $\pi(b)\not\in\Z\cdot v_0$. (For $b\in\Z\cdot v_0$, one has $\DM(\phi(x,y,b))=2$.) 

Thus, for generic $b$ the (unique) generic type $p(x,y)$ of $\phi(x,y,b)$ is strongly minimal and not uniformly 
of Morley degree 1. In particular, if $q(x)$ denotes the generic 1-type in $T$, $p$ and $q$ are non-orthogonal strongly minimal 
types, $q$ is uniformly of Morley degree 1 and $p$ is not. 
\end{example}

\begin{lemma}\label{lem:bddDegree}
  Let $T$ be almost $\aleph_1$-categorical, and let $\phi(x,z)$ be a formula.
  There is $N\in\N$ such that $\DM(\phi(x,b))\leq N$ for all $b$.
\end{lemma}

\begin{proof}
  Adding parameters to the language if necessary, we may assume that there are
  strongly minimal sets $(D_i)_{1\leq i\leq k}$ defined over $\emptyset$ such
  that every non-algebraic type is non-orthogonal to one of the $D_i$'s.

  \begin{claim*}
    Every global type $p(x)$ is a generic type in some $\phi(x,b)$,
    where $\phi(x,z)$ is an $\L$-formula for which the statement of the lemma is
    true.
  \end{claim*}

  We prove the claim by induction on $r=\RM(p)$, the case $r=0$ being
  trivial.

  Let $p$ be strongly minimal. By assumption, $p\not\perp D_i$ for some $i$.
  Let $D_i$ be defined by $\psi_i(y)$. Since $p(x)$ is a global type, there is
  $\phi(x,b)\in p$ and a finite-to-finite correspondence $\chi(x,y,c)$ between
  $\phi(x,b)$ and $\psi_i(y)$. Let $m,n\in\N$ such that $\chi(x,y,c)$
  generically induces an $m$-to-$n$ correspondence (i.e.\ outside a finite
  set). We may assume that $b=c$ and that $\chi(x,y,b')$ generically induces
  an $m$-to-$n$ correspondence between $\phi(x,b')$ and $\psi_i(y)$ whenever
  $\phi(x,b')$ is consistent. Since $\psi_i(y)$ is strongly minimal,
  $\DM(\phi(x,b'))\leq m$ for all $b'$. This proves the case $r=1$.

  For the induction step, we argue as in the proof of Lemma
  \ref{lem:charDMPsm}. Let $\RM(p)=r+1$. Replacing $p$ by a finite cover, we
  may assume that $p(x)=p(x_1x_2)$ and, by induction, that there is an
  $\L$-formula $\phi(x_2,x_1,z)$ and $b\in M$ such that

  \begin{itemize}
    \item $p$ is the unique generic type in $\phi(x_2,x_1,b)$ over $M$;
    \item $\RM(\phi(x_2,x_1,b'))=r+1$ whenever it is consistent;
    \item $\psi(x_1,b')=\exists x_2\phi(x_2,x_1,b')$ is of Morley rank 1 and
      degree $\leq N_1$ whenever it is consistent;
    \item whenever $\phi(x_2,a_1',b')$ is consistent, it is of Morley rank $r$
      and degree $\leq N_2$.
  \end{itemize}

  Then $\phi(x_2,x_1,b')$ is of Morley degree $\leq N_1 N_2$ for every $b'$,
  so the claim is proved.

  \smallskip

  Now let $\phi(x,z)$ be given. For a parameter $b$ with $\DM(\phi(x,b))=d$,
  let $p_1,\ldots,p_d$ be the (global) generic types of $\phi(x,b)$. By the
  claim, for $1\leq i\leq d$, there exists a formula $\chi(x,z_i)$, a
  parameter $b_i$ and $N_i\in\N$ such that $p_i$ is generic in $\chi_i(x,b_i)$
  and $\DM(\phi(x,b_i'))\leq N_i$ for every $b_i'$. Since Morley rank is
  definable, there is $\theta(y)\in\tp(b)$ such that whenever
  $\models\theta(b')$, there are $b_1',\ldots,b_d'$ with
  $\RM\left(\phi(x,b')\wedge\neg\bigvee_{i=1}^d\chi_i(x,b'_i)\right)<
  \RM\left(\phi(x,b')\right)=\RM\left(\chi_i(x,b'_i)\right)$
  for $i=1,\ldots,d$. So $\DM(\phi(x,b'))\leq\sum_{i=1}^d N_i$ if
  $\models\theta(b')$.

  Since such a formula exists for every $b$, we are done by compactness.
\end{proof}

Now let $T$ be a theory which is complete and has quantifier elimination in
some language $\cal L$, and let $\sigma\not\in\cal L$ be a new unary function
symbol. Consider $T_{\sigma}:=T\cup\{\sigma\text{ is an $\cal
L$-automorphism}\}$, a theory in the language $\cal L\cup\{\sigma\}$. Denote
by $TA$ the model companion of $T_{\sigma}$ if it exists.

In the proof of the following result, we proceed as in \cite{HaHr07}, where
Corollary \ref{cor:HaHiHr} is shown for $T$ strongly minimal.

\begin{proposition}\label{prop:finitecover}
  Let $T$ be stable, and assume $TA$ exists. Then, for strongly minimal types
  in $T$, being uniformly of Morley degree 1 is invariant under
  non-orthogonality.
\end{proposition}

\begin{proof}
  Let $p,q$ be strongly minimal types such that $p\not\perp q$ and $q$ is
  uniformly of Morley degree 1. We have to show that $p$ is uniformly of
  Morley degree 1. As in the proof of Lemma \ref{lem:charDMPsm}, using Lemma
  \ref{lem:uniformMD1}, we may assume that $p=\tp(\beta /M)$ is a finite cover
  of $q=\tp(b/M)$, and even that $b$ is a subtuple of $\beta$.

  Let $\beta_1=\beta,\beta_2,\ldots,\beta_n$ be the $Mb$-conjugates of
  $\beta$, and put $\tuple{\beta}_1=(\beta_1,\ldots,\beta_n)$. We may assume
  that $p=\tp(\tuple{\beta}_1/M)$. Then, if
  $\tuple{\beta}_1,\ldots,\tuple{\beta}_r$ are the $Mb$-conjugates of
  $\tuple{\beta}_1$, any $\tuple{\beta}_i$ is of the form
  $$\tuple{\beta}_i=(\beta_{\tau_i(1)},\ldots\beta_{\tau_i(n)})$$ for some
  $\tau_i\in S_n$. Moreover, by construction, $$G:=\{\tau_i\,\mid\, 1\leq
  i\leq r\}\leq S_n$$ is a subgroup of $S_n$.

  Since $q$ is uniformly of Morley degree 1 by assumption, there are
  $\L$-formulas $\tilde{\phi}(\xtup,y)$, $\phi(x,y)$, a parameter $a\in M$, a
  formula $\theta(y)\in r(y)=\tp(a)$, and a definable function $\pi$ from the
  $\xtup$-sort to the $x$-sort satisfying the following properties:
  \begin{itemize}
    \item[(i)] $\tilde{\phi}(\xtup,a)$ and $\phi(x,a)$ are strongly minimal,
      with generic types $p$ and $q$, respectively;
    \item[(ii)] whenever $\models\theta(a')$, $\phi(x,a')$ is strongly
      minimal, and $\pi$ defines a surjection from $\tilde{\phi}(\xtup,a')$
      onto $\phi(x,a')$;
    \item[(iii)] whenever $\models\theta(a')$, denoting $\tilde{D}_{a'}$ and
      $D_{a'}$ the sets defined by $\tilde{\phi}(\xtup,a')$ and $\phi(x,a')$,
      resp., all fibres of the map $\pi:\tilde{D}_{a'}\twoheadrightarrow D_{a'}$
      are regular $G$-orbits for the definable action of $G$ (by permutation)
      on the $\xtup$-sort, i.e.\ if $b'\in D_{a'}$, then
      $\mid\!\!\pi^{-1}(b')\!\!\mid=|G|$ and if
      $\pi^{-1}(b')=\{\tuple{\beta}'_1,\ldots,\tuple{\beta}'_s \}$, letting
      $\tuple{\beta}'_1=(\beta'_1,\ldots,\beta'_n)$, any $\tuple{\beta}'_i$ is
      of the form $(\beta'_{\tau(1)},\ldots,\beta'_{\tau(n)})$ for some
      $\tau\in G$.
  \end{itemize}

  \begin{claim*}
    For any $\tau\in G$, there is an $\L$-formula $\theta_{\tau}(y)\in r(y)$
    implying $\theta(y)$ such that, in $TA$, the following implication holds:
    $$\theta_{\tau}(y)\wedge\sigma(y)=y\vdash\exists
    x\,\exists\xtup\left[\pi(\xtup)=x\wedge\sigma(x)=
    x\wedge\tilde{\phi}(\xtup,y)\wedge\bigwedge_{i=1}^n\sigma(x_i)=
    x_{\tau(i)}\right]$$
  \end{claim*}

  \begin{proof}[Proof of the claim]
    By compactness, it is enough to show that
    $$r(y)\cup\{\sigma(y)=y\}\vdash\exists
    x\,\exists\xtup\left[\pi(\xtup)=x\wedge\sigma(x)=x
    \wedge\tilde{\phi}(\xtup,y)
    \wedge\bigwedge_{i=1}^n\sigma(x_i)=x_{\tau(i)}\right].$$
    Let $(M',\sigma)\models TA$ and $a'\in M'$ such that $\sigma(a')=a'$ and
    $\models r(a')$. Then, $\tilde{D}_{a'}$ is strongly minimal (in $T$). Let
    $\tuple{\beta}'=(\beta'_1,\ldots,\beta'_n)$ be generic in $\tilde{D}_{a'}$
    over $M'$. Then
    $\tau\cdot\tuple{\beta}':=(\beta_{\tau(1)},\ldots,\beta_{\tau(n)})$ is
    also generic in $\tilde{D}_{a'}$ over $M'$. Thus,
    $\tp_{\L}(\tuple{\beta}'/M')=\tp_{\L}(\tau\cdot\tuple{\beta}'/M')$. Since
    $\sigma(a')=a'$, it follows that
    $\sigma\upharpoonright_{\acl(a')}\cup\{\beta'_i\mapsto\beta'_{\tau(i)}\,
    \mid\, 1\leq i\leq s\}$ is an elementary map. So, in some elementary
    extension of $(M',\sigma)$, there is such a tuple $\tuple{\beta}'$ such
    that $\sigma(\beta'_i)=\beta'_{\tau(i)}$ for all $i$. Moreover, by
    construction, $\pi(\tuple{\beta}')=\pi(\sigma(\tuple{\beta}'))$, so
    $\sigma(\pi(\tuple{\beta}'))=\pi(\tuple{\beta}')$. This proves the claim.
  \end{proof}

  Let $\tilde{\theta}(y)=\bigwedge_{\tau\in G}\theta_{\tau}(y)$. We will show
  that $\tilde{\phi}(\xtup,a')$ is strongly minimal whenever
  $\models\tilde{\theta}(a')$. This will finish the proof.

  From now on, we may proceed exactly as in \cite{HaHr07}.

  By (ii), $\RM(\tilde{D}_{a'})=1$. Let $\dtup$ and $\dtup'$ be generic in
  $\tilde{D}_{a'}$ over $a'$. We have to show that
  $\stp(\dtup/a')=\stp(\dtup'/a')$. Since $D_{a'}$ is strongly minimal by (i),
  $\stp(\pi(\dtup)/a')=\stp(\pi(\dtup')/a')$, so we may assume that
  $\pi(\dtup)=\pi(\dtup')=d$.

  By (iii), $G$ acts regularly on $\pi^{-1}(d)$, via
  $$\tau\cdot\dtup_\rho=\tau\cdot(\delta_{\rho(1)},\ldots,\delta_{\rho(n)})=\dtup_{\tau\rho} \,\,\,\,\text{Ê(for $\rho\in G$)}.$$
  In particular, there is $\rho'\in G$ such that $\rho'\cdot\dtup=\dtup_{\rho'}=\dtup'$.
  Consider $$H:=\left\{\tau\in
  G\,\mid\,\id\res_{\acl(a')}\cup\{\dtup\mapsto\tau\cdot\dtup\}\text{ is an
  elementary map}\right\}.$$ Note that $\tau$ is in $H$ iff there is an element $\rho\in G$ such that $\id\res_{\acl(a')}\cup\{\rho\cdot\dtup\mapsto(\tau\rho)\cdot\dtup\}$ is an elementary map  iff 
  $\id\res_{\acl(a')}\cup\{\rho\cdot\dtup\mapsto(\tau\rho)\cdot\dtup\,\mid\,
  \rho\in G\}$ is an elementary map. In particular, $H\leq G$ is a subgroup.

  Let $\id\neq\tau\in G$.  We may embed $(\acl(a'),\id)$ into
  some model $(M',\sigma)\models TA$. By the claim, there is $\tuple{\beta}^{\tau}\in
  \tilde{D}_{a'}$ such that
  $\sigma(\tuple{\beta}^{\tau})=\tau\cdot\tuple{\beta}^{\tau}=(\beta^\tau_{\tau(1)},\ldots,\beta^\tau_{\tau(n)})$. As
  $\tau\neq\id$, $\tuple{\beta}^{\tau}\not\in\acl(a')$, so
  $b^{\tau}=\pi(\tuple{\beta}^{\tau})$ is generic in $D_{a'}$ over $a'$. Let
  $\alpha:\acl(a'd)\cong\acl(a'b^{\tau})$ be an elementary map extending
  $\id\res_{\acl(a')}\cup\{d\mapsto b^{\tau}\}$. Clearly,
  $\sigma'=\alpha^{-1}\circ\sigma\circ\alpha$ is an elementary permutation of
  $\acl(a'd)$ fixing $\acl(a')\cup\{d\}$. Moreover, if
  $\alpha^{-1}(\tuple{\beta}^{\tau})=\tau'\cdot\dtup$, then $\alpha^{-1}(\rho\cdot\tuple{\beta}^{\tau})=(\tau'\rho)\cdot\dtup$ for every $\rho\in G$, and an easy
  calculation shows that $\tau'\tau\tau'^{-1}\in H$; indeed,
  $$(\alpha^{-1}\circ\sigma\circ\alpha)(\dtup_{\tau'})=
  \alpha^{-1}(\sigma(\tuple{\beta}^{\tau}))=
  \alpha^{-1}(\tau\cdot\tuple{\beta}^{\tau})=(\tau'\tau)\cdot\dtup=\left( (\tau'\tau\tau'^{-1})\tau'\right)\cdot\dtup=\tau'\tau\tau'^{-1}\cdot\dtup_{\tau'}.$$

  Thus, all conjugacy classes in $G$ are represented in $H$, showing that
  $H=G$ (see \cite{HaHr07}). This completes the proof.
\end{proof}

\begin{corollary}\label{cor:TADMP}
  Let $T$ be stable, and assume $TA$ exists. Let $X$ be definable and
  $T'=\mathrm{Th}(X_{ind})$ be the theory of the induced structure on $X$.
  Assume $T'$ is almost $\aleph_1$-categorical.

  Then $T'$ has the DMP. In particular, any group of finite Morley rank
  interpretable in $T$ (considered with the full induced structure) has the
  DMP.
\end{corollary}

\begin{proof}
  It is easy to see that if $T'_B$ has the DMP for some parameter set $B$,
  then $T'$ has the DMP (see \cite{Hru92}). We may thus assume that every
  strongly minimal type $p'$ in $T'$ is non-orthogonal to some strongly
  minimal type $q'$ over $\emptyset$. Trivially, such a $q'$ is uniformly of
  Morley degree 1, and so $p'$ is uniformly of Morley degree 1 by Proposition
  \ref{prop:finitecover}. This shows that $T'$ has the DMP, by Lemma
  \ref{lem:charDMPsm}.
\end{proof}

\begin{corollary}\label{cor:HaHiHr}
  Let $T$ be almost $\aleph_1$-categorical, e.g. $T=\mathrm{Th}(G)$ for $G$ a
  group of finite Morley rank. Then $TA$ exists if and only if $T$ has the
  DMP.
\end{corollary}

\begin{proof}
  Using Fact \ref{fact:PillayPong}, it is easy to see that $TA$ exists if $T$
  has the DMP (see \cite{ChPi98}). The other direction is the previous
  corollary.
\end{proof}

\section{Generic automorphisms of compact complex manifolds}
In this section, we apply the results of Section~\ref{S:TA-DMP} to compact
complex manifolds, deducing in particular that definable groups have
the DMP. We use the result of Radin \cite{Rad04}
that topological irreducibility is definable in families. Unlike in algebraic
geometry, it is in general not true in compact complex manifolds that
irreducibility implies Morley degree 1, so DMP does not follow immediately.
Rather, we note that definability of irreducibility in the theory of a
Noetherian topological structure (defined below) suffices, by a
straightforward generalisation of the well-known existence of geometric axioms
for $\ACFA$, to prove axiomatisability of generic automorphisms; the results
of Section~\ref{S:TA-DMP} then apply.

We take the following definitions from \cite{ZilberZariski}.

\begin{definition}
  \begin{itemize}
    \item A \defnstyle{topological structure} consists of a set $S$ and a
      topology on each $S^n$ such that
      \begin{enumerate}[(i)]
	\item The co-ordinate projection maps $\pr : S^n \maps S^m$ are
	  continuous;
	\item The inclusion maps
	  \[ \iota : S^m \maps S^n; (x_1,\ldots,x_m) \mapsto
	  (x_1,\ldots,x_m,c_{m+1},\ldots,c_n) \]
	  are continuous.
	\item The diagonal $\Delta\subseteq S^2$ is closed.
	\item Singletons $\{\atup\}\subseteq S^n$ are closed.
      \end{enumerate}
    \item A topological structure is \defnstyle{Noetherian} iff the closed
      sets satisfy the descending chain condition.

    \item A topological structure is \defnstyle{$\omega_1$-compact} iff
      whenever $(C_i)_{i\in I}$ is a countable set of closed sets with the
      property that $\bigcap_{i\in I_0} C_i \neq \emptyset$ for any finite
      $I_0\subseteq I$, then $\bigcap_{i\in I} C_i \neq \emptyset$.

    \item We consider a topological structure $S$ as a first-order structure
      in the language having a predicate for each closed set. Since singletons
      are closed, models $S'$ of $\Th(S)$ are precisely elementary extensions
      of $S$. We consider such $S'$ as topological structures, with closed
      sets the fibres with respect to co-ordinate projections of the
      interpretations in $S'$ of the closed sets of $S$. It is easy to see
      that if $S$ is $\omega_1$-compact and Noetherian, then $S'$ is also
      Noetherian.

    \item A \defnstyle{constructible} set is a finite boolean combination of
      closed sets, so $\Th(S)$ has quantifier elimination iff every definable
      set is constructible. A constructible set is \defnstyle{irreducible} iff
      it is not the union of two relatively closed proper subsets.

    \item Suppose now that $T$ is the theory of an $\omega_1$-compact
      Noetherian topological structure $S$ with quantifier elimination.

      If $X$ is an irreducible constructible set defined over a model
      $S'\elres S$, the \defnstyle{generic type} over $S'$ of $X$ is the
      complete type
      \[ p_X^{S'}(x) = \{ x\in O \;|\; O \; \textrm{a relatively open subset of
      $X$ defined over $S'$} \} .\]

      Conversely, if $S'' \elres S'$ and $a\in S{''}^n$, the \defnstyle{locus}
      of $a$ over $S'$ ($\locus(a/S')$) is the smallest closed set defined
      over $S'$ containing $a$.

      We say that \defnstyle{irreducibility is definable} in $T$ iff for any
      $S'\models T$ and any constructible $C(x;y) \subseteq S{'}^{n+m}$, the
      set
      \[ \{ y \;|\; \textrm{$C(S',y)$ is irreducible} \} \subseteq S{'}^m \]
      is definable over $S'$.
  \end{itemize}
\end{definition}

A compact complex manifold $X$ can be considered as an $\omega_1$-compact
Noetherian topological structure, where the closed sets are the complex
analytic subsets.

It is a result of Zilber \cite[Theorem 3.4.3]{ZilberZariski} that the
structure has quantifier elimination and finite Morley rank.

\begin{proposition}\label{prop:TANoeth}
  Let $T$ be the theory of an $\omega_1$-compact Noetherian topological
  structure with quantifier elimination, in which irreducibility is definable.
  Then $TA$ exists.
\end{proposition}
\begin{proof}

  This is a straightforward generalisation of the case of algebraically closed
  fields.

  $TA$ is axiomatised by
  \begin{enumerate}[(i)]
    \item $(S,\sigma)\models T_\sigma$
    \item If $U$ and $V\subseteq U\times\sigma(U)$ are closed irreducible sets
      in $S$ such that $V$ projects generically onto $U$ and $\sigma(U)$ (i.e.
      $U$ is the closure of $\pr_1(V)$ and $\sigma(U)$ is the closure of
      $\pr_2(V)$), and if $W\subsetneq V$ is proper closed in $V$, then there
      exists a point $(a,\sigma(a))\in V\setminus W$.
  \end{enumerate}

  By definability of irreducibility, (ii) is indeed first-order expressible.

  Just as in \cite[(1.1)]{ChatHrushACFA}, we find that $TA$ axiomatises the
  class of existentially closed models of $T_\sigma$; indeed: by uniqueness of
  generic types of irreducible sets, any existentially closed model satisfies
  (ii). For the converse, by quantifier elimination and closedness of equality
  it suffices to see that if $(S,\sigma)\models TA$ and $(S',\sigma)\models
  T_\sigma$ is an extension (so $S'\elres S$), and if $V'$ is closed
  irreducible in $S$ and $W$ is a proper closed subset, then if there exists
  $a\in S'$ such that $(a,\sigma(a))\in V'(S')\setminus W'(S')$, then already
  there exists such an $a\in S$. But indeed, this follows from (ii) on taking
  $V := \locus((a,\sigma(a))/S)$, taking $U := \locus(a/S)$, and taking
  $W:=W'\cap V$.
\end{proof}

In particular, then

\begin{corollary}\label{cor:CCMA}
  Let $T$ be the theory of a compact complex manifold. Then $TA$ exists.
\end{corollary}

Combining with Corollary~\ref{cor:TADMP}, we obtain
\begin{corollary}\label{cor:CCMDMP}
  Let $T$ be the theory of a compact complex manifold, and suppose $T$ is
  almost $\aleph_1$-categorical. Then $T$ has the DMP.
\end{corollary}

For the case of unidimensional $T$, this was deduced by a more direct method
from definability of irreducibility by Radin (\cite{Rad04}).

\begin{remark}
  For clarity, in this section we have worked one sort at a time; but it is
  entirely straightforward to generalise to many-sorted topological
  structures, and in particular to the structure $\mathfrak{CCM}$ which has a
  sort for each compact complex manifold and predicates for complex analytic
  subsets of powers of the sorts.
\end{remark}

\begin{question}
  Let $(G;\cdot)$ be a Zariski group, i.e. a Noetherian Zariski structure (in
  the sense of \cite{ZilberZariski}) with a group operation $\cdot$ whose
  graph is closed in $G^3$. Does $G$ necessarily have definability of
  irreducibility (and hence, by Proposition~\ref{prop:TANoeth} and
  Corollary~\ref{cor:HaHiHr}, Morley degree)? What if $G$ is assumed to be
  presmooth?
\end{question}

\section{Kummer genericity in divisible abelian groups of finite Morley}
  \label{S:fMR}

In this section, let $S$ be a definable divisible abelian group of finite
Morley rank $d$ (in some stable theory $T$). In particular, since $S$ is
divisible, it is connected, and so $\DM(S)=1$.

By rank considerations, the $n$-torsion subgroup $S[n]$ is finite for every
$n$. We continue to denote by $T = \invlim_n S[n]$ the projective limit of the
torsion subgroups.

Note that in general there might be a proper definable subgroup of $S$ that
contains all torsion elements of $S$. By the DCC for definable subgroups in
$S$, there is a smallest such group, the \emph{definable hull} of the torsion
subgroup of $S$. We denote it by $d(S_{\mathrm{tors}})$.

We use the following version of Zilber's Indecomposability Theorem for
types in commutative groups of finite Morley rank.

\begin{lemma}\label{lem:typeZIndec}
  Let $p$ be a strong type over $A$ extending $S$. Let $p^{(d)}$ be the type
  of an independent $d$-tuple of realisations of $p$, and let $q := \Sigma_*
  p^{(d)}$ be its image under the summation map $\Sigma : S^d\maps S$. Then
  $q$ is the generic type of an $\acleq(A)$-definable coset of a connected
  definable subgroup $H=H(p)$ of $S$. In particular, $H(p)$ is equal to the
  stabiliser of $q$.
\end{lemma}
\begin{proof}
  It is easy to see that $n\mapsto \RM(\Sigma_* p^{(n)})$ is an increasing
  function, and that $\RM(\Sigma_* p^{(n)})=\RM(\Sigma_* p^{(n+1)})$ implies
  that $\RM(\Sigma_* p^{(n)})=\RM(\Sigma_* p^{(n+m)})$ for all $m\in\N$. Thus,
  the Morley rank of $q := \Sigma_* p^{(d)}$ is maximal among $\RM(\Sigma_*
  p^{(n)})$.

  Now let $a,b$ be independent realisations of $q$, and let $c:=-a-b$. Using
  additivity of Morley rank, we get that $a,b,c$ is pairwise independent. The
  result then follows from \cite[Theorem 1]{Zie04}.
\end{proof}

\begin{definition}
  Let $p$ be a strong type extending $S$, with $H(p)$ as in
  Lemma~\ref{lem:typeZIndec}.
  \begin{itemize}
    \item $p$ is called \emph{free} if and only if $H(p)=S$;
    \item $p$ is called \emph{Kummer-generic} if there is only one
      strong completion of the partial long type $\{ [m]x_{mn}=x_n  \mid
      m,n\in\N \}\cup p(x_1)$, and $p$ is called \emph{almost Kummer-generic}
      if this type has only finitely many strong completions.
    \item for $X$ some definable subset of $S$ of Morley degree 1, we say $X$
      is \emph{free} (respectively \emph{(almost) Kummer-generic}), if the
      unique generic type of $X$ is free (respectively (almost) Kummer-generic).
  \end{itemize}
\end{definition}

It follows from Lemma \ref{lem:typeZIndec} that $p$ is free iff the sum of $d$
independent realisations of $p$ is generic in $S$. But note that in contrast
to the case where $X$ is an irreducible subvariety of a semiabelian variety
$S$, for $X$ some definable subset of $S$ of Morley degree 1,
$\RM(\Sigma(X^d))=d$ does not in general imply that $X$ is free.

\begin{lemma}\label{lem:free-Kg}
  \begin{enumerate}
    \item Let $X\subseteq S$ be definable such that $\RDM(X)=(k,1)$.
      \begin{enumerate}
	\item $X$ is free iff $\{a\in
	  S\,\mid\,\RM(\Sigma^{-1}(a)\cap X^d)=dk-d\}$ is a generic subset of
	  $S$.
	\item $X$ is Kummer-generic iff $\DM([n]^{-1}(X))=1$ for
	  all $n\in\N$.
      \end{enumerate}
    \item Let $p$ be a strong type extending $S$.
      \begin{enumerate}
	\item If $p'$ is a translate of a non-forking extension of $p$, then $p'$ is free
	  (respectively (almost) Kummer-generic) iff $p$ is free
	  (respectively (almost) Kummer-generic).
	\item There exists a translate $p'$ of a non-forking extension of
	$p$ such that $p'$ extends $H(p)$ and $\Sigma_* p'^{(d)}$ is the
	  generic type of $H(p')=H(p)$. Moreover, the following are
	  equivalent:
	  \begin{itemize}
	    \item $p$ is (almost) Kummer-generic;
	    \item $p'$ is (almost) Kummer-generic;
	    \item $H(p)\geq S_{tors}$ and $p'$ is (almost) Kummer-generic in
	      the group $H(p)$.
	      \end{itemize}
      \end{enumerate}
  \end{enumerate}
\end{lemma}

\begin{proof}
  1.a) Let $p$ be the generic type of $X$. Then $p^{(d)}$ is the unique
  generic type of $X^d$, and $X$ is free iff $\Sigma_*p^{(d)}$ is the generic
  type of $S$ iff for generic $a$ in $S$, $\RM(\Sigma^{-1}(a)\cap X^d)=dk-d$.
  (This uses the additivity of the Morley rank.) The result follows.

  1.b) Note that since $[n]$ is finite-to-one, an element $a\in [n]^{-1}(X)$
  is generic iff $na$ is generic in $X$ iff $na\models p$. From this, one may
  conclude.

  2.a) Left to the reader.

  2.b) Let $(a_1,a_2)\models p^{(2)}\mid M$. It is routine to check that
  $p':=\tp(a_2-a_1/Ma_1)$ is as claimed. Since $p'$ is a translate of (a
  non-forking extrension of) $p$, the first two items are equivalent by 2.a).

  Now suppose $p'$ is Kummer-generic in $S$. Then, clearly $p'$ is
  Kummer-generic in $H(p')=H(p)$ as well. Moreover, if $\zeta\in
  S_{tors}\setminus H(p)$, with $n\zeta=0$, and $a'\models p'$, there is
  $c'\in H(p)$ such that $[n]c'=a'$, since $H(p)$ is connected and thus
  divisible. Then, $\tp(c'+\zeta/M)\neq \tp(c'/M)$, since $c'+\zeta\not\in
  H(p)$. This shows that $H(p)\geq S_{tors}$.

  Conversely, assume that  $H(p)\geq S_{tors}$ and $p'$ is Kummer-generic in
  $H(p)$. Then, for every $b\in H(p)$ and every $n\geq1$,
  $[n]^{-1}(b)\subseteq H(p)$, as $S[n]\subseteq H(p)$. So Kummer-genericity
  of $p'$ in $S$ follows.

  We now give the argument for almost Kummer-genericity. If $N\geq1$ and $q$
  is a completion of $\pi(x):=p([N]x)$, then $H(q)=H(p)$. (Indeed, in a
  totally transcendental divisible abelian group, if $p=[N]_*q$, it is easy to
  see that $\Stab(p)=N\Stab(q)$. Moreover, $\Stab(q)=N\Stab(q)$ since
  $\Stab(q)$ is connected. Thus, $\Stab(p)=\Stab(q)$.) Now $p(x)$ is almost
  Kummer generic iff there is $N\geq1$ such that every completion $q$ of the
  partial type $\pi(x):=p([N]x)$ is Kummer-generic. So we conclude by the
  result for Kummer-genericity.
\end{proof}

\begin{theorem}\label{thm:fmr1}
  Let $p$ be a strong type over $A$ extending $S$.

  Assume $p$ is free. Then $p$ is almost Kummer-generic.

  Moreover, $p$ is almost Kummer-generic if and only if $H(p)\geq S_{tors}$.
\end{theorem}

\begin{proof}
  Assume $p$ is free. Replacing $p$ by a non-forking extension if necessary,
  we may assume that $A=M\models T$, and that $S$ is defined over $M$.

  Let $\btup\models p^{(d)}$, and put $a=\Sigma(\btup)$. Since $p$ is free,
  $a$ is generic in $S$ over $M$. Let $P$ be the set of types over
  $\acleq(Ma)$ extending $\tp(\btup/Ma)$. Let $X$ be an $M$-definable set with
  unique generic type $p$. By Fact \ref{fact:Lascar}, Morley rank is finite
  and additive. It follows that $\tp(\btup/Ma)$ is the unique type in
  $X^d\cap\Sigma^{-1}(a)$ over $Ma$ of maximal Morley rank, and so $P$ is
  equal to the set of generic types in $X^d\cap\Sigma^{-1}(a)$ over
  $\acleq(Ma)$. In particular, $P$ is a finite set, say
  $P=\{q_1,\ldots,q_m\}$, where $m$ is the Morley degree of
  $X^d\cap\Sigma^{-1}(a)$.

  Let $\Gt=\Gal(Ma)$ be the (absolute) Galois group of $Ma$, i.e.\ the set of
  elementary permutations of $\acleq(Ma)$ fixing $Ma$ pointwise. The group
  $\Gt$ acts transitively on $P$. The proof of Claim \ref{clm:maingal} in the
  proof of Proposition \ref{prp:main} goes through in this context (with
  $\Stab_{\Gt}(V)$ replaced by $\Stab_{\Gt}(q_1)$), and it yields the theorem
  exactly as in the semiabelian case.

  The ``moreover'' clause follows by Lemma~\ref{lem:free-Kg}(2)(b).
\end{proof}

We obtain an analogous result in the type-definable case. We refer to
\cite[Section 2.3]{BeBoPiSSharp} for the definition of relative Morley rank.

\begin{theorem}\label{thm:type-definableThm1}
  Theorem \ref{thm:fmr1} also holds in the case that $S$ is a type-definable
  divisible abelian group of finite relative Morley rank.
\end{theorem}

\begin{proof}
  Say $S$ is of finite relative Morley rank $d$, and let $p$ be a strong type
  over $A=M$ extending $S$. Assume $p$ is free (i.e. $\Sigma_*p^{(d)}$ is the
  generic type of $S$).

  Let $X\subseteq S$ be a relatively definable subset of $S$ such that $p$ is
  the unique generic type of $X$. Then, for $a$ generic in $S$ over $M$,
  $Y=X^d\cap\Sigma^{-1}(a)$ is a relatively $Ma$-definable subset of $S^d$, so
  in particular the set $P$ of generic types in $Y$ over $\acleq(Ma)$ is
  finite (equal to the relative Morley degree of $Y$). We finish as in the
  proof of Theorem \ref{thm:fmr1}.
\end{proof}

We now state a uniform version of Theorem \ref{thm:fmr1}.

\begin{theorem}\label{thm:uniformThm1}
  Suppose $T$ is almost $\aleph_1$-categorical and $(S_t)_{t\in\cal T}$ is a
  uniformly definable family of divisible abelian groups in $T$. Let
  $(X_t)_{t\in\cal T}$ be a definable family of Morley degree 1 sets, with
  $X_t$ a free subset of $S_t$ for every $t\in\cal T$. Let $p_t$ be the
  generic global type of $X_t$.

  Then there is $N\in\N$ such that for every $t\in\cal T$, the partial long
  type $\{ [m]x_{mn}=x_n  \mid m,n\in\N \}\cup p_t(x_1)$ has at most $N$
  completions to global types.
\end{theorem}

\begin{proof}
  Since Morley rank is definable in $T$ (Fact \ref{fact:PillayPong}), we may
  assume that $\RM(S_t)=d$ for all $t\in\cal T$. Let
  $\Sigma_t:S_t^d\rightarrow S_t$ be the summation map. We infer from Lemma
  \ref{lem:bddDegree} that $\DM(X_t^d\cap\Sigma_t^{-1}(a))$ is bounded. The
  statement follows, by the proof of Theorem \ref{thm:fmr1}.
\end{proof}

\begin{theorem}\label{thm:fmr2}
  Suppose $T$ is $\omega$-stable with the DMP and $(S_t)_{t\in\cal T}$ is a
  uniformly definable family of divisible abelian groups of finite Morley rank
  in $T$. Let $(X_t)_{t\in\cal T}$ be a definable family of Morley degree 1
  sets, with $X_t$ a subset of $S_t$ for every $t\in\cal T$. Let $p_t$ be the
  generic global type of $X_t$.
  Then
  \begin{enumerate}[(i)]
    \item $\{ t\in\cal T \,|\, \textrm{$p_t$ is free and Kummer-generic in
      $S_t$} \}$ is definable;
    \item the set of $t$ such that some translate of $p_t$ is free and
      Kummer-generic in $H(p_t)$ is definable.
    \item Suppose that the family $(S_t)_{t\in\cal T}$ is constant (equal to
      $S$). Then the set $\{ t\in \cal T \,|\, \textrm{$p_t$ is Kummer-generic
      in $S$} \}$ is definable.
  \end{enumerate}
\end{theorem}

\begin{proof}
  Adding parameters to the language if necessary, we may assume the two
  families are defined over $\emptyset$.
  \begin{enumerate}[(i)]
    \item Using the DMP, we may assume that $\RM(S_t)$ is constant, say equal
      to $d$. We infer from Lemma \ref{lem:free-Kg} that freeness of $X_t$ (in
      $S_t$) is a definable condition in $t$, and so we may in addition assume
      that all $X_t$ are free in $S_t$. By the DMP,
      $\DM(X_t^d\cap\Sigma_t^{-1}(a))$ is bounded. As in the proof of Theorem
      \ref{thm:uniformThm1}, we find $N\in\N$ such that for every $t\in\cal
      T$, the partial long type $\{ [m]x_{mn}=x_n \mid m,n\in\N \}\cup
      p_t(x_1)$ has at most $N$ completions to global types. It follows that
      $X_t$ is Kummer generic in $S_t$ iff $\DM([l]^{-1}(X_t))=1$
      for all primes $l\leq N$. The latter is a definable condition by the
      DMP.

    \item By Lemma \ref{lem:typeZIndec}, $H(p_t)$ is equal to the stabiliser
      of $\Sigma_{*}(p_t^{(d)})$, so $(H(p_t))_{t\in\cal T}$ is a definable
      family of (divisible) subgroups of  the $S_t$. It is clear that one
      translate of $p_t$ in $H(p_t)$ is Kummer-generic in $H(p_t)$ if and only
      if all translates of $p_t$ in $H(p_t)$ are Kummer-generic in $H(p_t)$. We
      are done by part (i), since we may consider the family
      $\{Y_{s,t}=(s+_tX_t)\cap H(p_t) \,|\, \,\RM(X_t)=\RM(Y_{s,t}\}$.

    \item By Lemma~\ref{lem:free-Kg}(2)(b), $p_t$ is Kummer generic (in $S$)
      iff a translate of $p_t$ is Kummer generic in $H(p_t)$ and
      $H(p_t)\geq S_{\mathrm{tors}}$. The latter condition is equivalent
      to $H(p_t)\geq d(S_{\mathrm{tors}})$, and so we conclude by (ii).
      \qedhere
  \end{enumerate}
\end{proof}


\section{Interesting examples}
\providecommand{\Ssharp}{S^{\#}}

Finally, using the results of previous sections, we give examples of groups to
which the results of Section~\ref{S:fMR} apply:

\begin{examples}\label{examples}
  \begin{enumerate}[(i)]
    \item Commutative divisible algebraic groups: As mentioned in the
      introduction, Theorems~\ref{thm:fmr1} and \ref{thm:fmr2} slightly
      generalise Theorems~\ref{thm:ag1} and \ref{thm:ag2}, giving results for
      arbitrary commutative divisible algebraic groups rather than just
      semiabelian varieties.
    \item Divisible abelian groups of finite Morley rank in $DCF_0$:
      groups definable in $DCF_0$ have the DMP - this follows from
      Corollary~\ref{cor:TADMP} and the result of Hrushovski that $DCF_0A$
      exists (see \cite{Bus07}). Therefore both Theorems~\ref{thm:fmr1} and
      \ref{thm:fmr2} apply.
    \item Divisible abelian groups interpretable in compact complex
      manifolds: any such group is of finite MR, and by
      Corollary~\ref{cor:CCMDMP} it has the DMP. So again,
      Theorems~\ref{thm:fmr1} and \ref{thm:fmr2} apply.
    \item
      Type-definable groups in the theory $SCF_{p,e}$ of separably closed
      fields of a fixed characteristic $p$ and inseparability degree $e$:
      by \cite{BoDe}, the commutative divisible type-definable groups in
      $SCF_{p,e}$ are (up to definable isomorphism) precisely those of the
      form $\Ssharp := \bigcap_n p^n S$ for $S$ a semiabelian variety.

      \cite[Proposition 3.23]{BeBoPiSSharp} shows that for certain semiabelian
      varieties $S$, in particular those which are split (i.e. a product of a
      torus and an abelian variety), $\Ssharp$ has finite relative MR. Hence
      Theorem~\ref{thm:type-definableThm1} applies.
  \end{enumerate}
\end{examples}

It is known that $SCF_{p,e}A$ exists (\cite{ChatzSCFA}), so it is tempting in
the context of the last example to try to deduce DMP for $\Ssharp$ by
appeal to Corollary~\ref{cor:TADMP}, for those $\Ssharp$ having relative
quantifier elimination; however, it is not clear that existence of
$SCF_{p,e}A$ implies axiomatisability of the generic automorphism for the
induced structure on $\Ssharp$. In an earlier version of the paper, we put
a general version of this as a question, asking:

\begin{question*}
  If $T$ is a stable theory for which $TA$ exists, and if $X$ is a
  type-definable set with relative QE, does $Th(X_{ind})A$ necessarily
  exist?
\end{question*}

The general answer to this question is no, as is shown by the following
example due to Nick Ramsey.

\begin{example}
  Let $\L$ be the language consisting of a binary relation $E$ and constant
  symbols $c_{i,n}$ for $1\leq i\leq n<\omega$. Let $T$ be the $\L$-theory
  axiomatised by:
  \begin{itemize}
    \item $E$ is an equivalence relation with infinitely many classes, and
      every class is infinite;
    \item the $c_{i,n}$ are pairwise distinct, and $E(c_{i,n},c_{j,m})$ holds
      iff $n=m$.
  \end{itemize}
  $T$ is complete with QE, has finite additive Morley rank and the DMP. So
  $TA$ exists (see \cite{ChPi98}).

  For $n\geq1$, consider $\phi_n(x):= E(x,c_{n,n})\rightarrow \bigvee_{i\leq
  n}x=c_{i,n}$, and let $X$ be the type-definable set given by
  $\{\phi_n\,\mid\, n\geq1\}$. The induced structure on $X$ is that of an
  equivalence relation with exactly one equivalence class for every $n\geq1$
  whose elements are named by constants. In particular, $X$ has relative QE
  and $X_{ind}$ is stable with the finite cover property. By an observation of
  Kudaibergenov (see \cite[Fact 3.5]{Kik00}), it follows that $Th(X_{ind})A$ does not exist.

  \smallskip
  In a similar fashion, Ramsey obtains an example where $X$ has relative QE,
  $X_{ind}$ is nfcp but $Th(X_{ind})A$ does not exist.
\end{example}

Finally, let us mention (see the following remark) that there is another way
of viewing the notion of Kummer-genericity for subvarieties of a semiabelian
variety in a more abstract manner. This reformulation shows that the main
ideas of the present paper are not bound to the presence of an underlying
(commutative) group. It is then reasonable to ask whether our results extend
to more general expansions of $\omega$-stable (finite rank) theories, e.g. in
the context of Shimura varieties. In the model-theoretic study of the
$j$-function, the relationship between affine $n$-space (with its structure as
an algebraic variety) and its reduct given by the Hecke correspondences plays
an important role, and there are strong similarities to the semiabelian
context. Forthcoming work of Adam Harris is expected to be relevant to this
case.

\begin{remark}
  Suppose $S$ is a semiabelian variety defined over $K=K^{alg}$. Let $T_1$ be
  the theory of the structure $S_1$ with underlying set $S(K)$, and with a
  predicate for every algebraic subvariety (defined over $K$) of some
  cartesian power of $S$. Let $T_0$ be the reduct of $T_1$, with predicates
  only for algebraic subgroups of cartesian powers of $S$. Let $T_i$ be ${\cal
  L}_i$-theories, for $i=0,1$, and put $S_0=S_1\res_{{\cal L}_0}$. (Note that
  $S_0$ is an abelian structure, so in particular $T_0$ is one-based.)

  Let $X\subseteq S$ be an irreducible variety defined over
  $L=L^{alg}\supseteq K$, and let $p_1$ be the generic ${\cal L}_1$-type of
  $X$ over $S(L)$. Then $X$ is free iff $p_1\res_{{\cal L}_0}$ is the (unique)
  generic ${\cal L}_0$-type over $S(L)$, and $X$ is Kummer-generic iff it is
  free and for $b$ realising $p_1$, the natural map of absolute Galois groups
  $\Gal_{T_1}(Lb)\maps \Gal_{T_0}(Lb)$ is a surjection.
\end{remark}

\bibliography{kg}

\begin{thebibliography}{BHMW09}

\bibitem[Bay09]{BaysThesis}
Martin Bays.
\newblock {\em Categoricity Results for Exponential Maps of 1-Dimensional
  Algebraic Groups \& Schanuel Conjectures for Powers and the CIT}.
\newblock PhD thesis, Oxford University, 2009.
\newblock URL: \url{http://www.math.mcmaster.ca/~mbays/dist/thesis}.

\bibitem[BBP09]{BeBoPiSSharp}
Franck Benoist, Elisabeth Bouscaren, and Anand Pillay.
\newblock Semiabelian varieties over separably closed fields, maximal divisible
  subgroups, and exact sequences.
\newblock 2009.
\newblock arXiv:0904.2083v1 [math.LO].

\bibitem[BD02]{BoDe}
E.~Bouscaren and F.~Delon.
\newblock Minimal groups in separably closed fields.
\newblock {\em J. Symbolic Logic}, 67(1):239--259, 2002.

\bibitem[Bel73]{Belegradek}
O.~V. Belegradek.
\newblock Almost categorical theories.
\newblock {\em Sibirsk. Mat. \v Z.}, 14:277--288, 460, 1973.

\bibitem[BHMW09]{BaHiMaWa09}
Andreas Baudisch, Martin Hils, Amador {Martin Pizarro}, and Frank~O. Wagner.
\newblock Die b\"ose {F}arbe.
\newblock {\em J. Inst. Math. Jussieu}, 8(3):415--443, 2009.

\bibitem[{Bus}07]{Bus07}
Ronald {Bustamante Medina}.
\newblock Differentially closed fields of characteristic 0 with a generic
  automorphism.
\newblock {\em Revista de Matem\'{a}tica: Teor\'{i}a y Aplicaciones},
  14:81--100, 2007.

\bibitem[BZ11]{BZCovers}
Martin Bays and Boris Zilber.
\newblock Covers of multiplicative groups of algebraically closed fields of
  arbitrary characteristic.
\newblock {\em Bull. Lond. Math. Soc.}, 43(4):689--702, 2011.

\bibitem[CH99]{ChatHrushACFA}
Zo{\'e} Chatzidakis and Ehud Hrushovski.
\newblock Model theory of difference fields.
\newblock {\em Trans. Amer. Math. Soc.}, 351(8):2997--3071, 1999.

\bibitem[Cha01]{ChatzSCFA}
Zo{\'e} Chatzidakis.
\newblock Generic automorphisms of separably closed fields.
\newblock {\em Illinois J. Math.}, 45(3):693--733, 2001.

\bibitem[CP98]{ChPi98}
Z.~Chatzidakis and A.~Pillay.
\newblock Generic structures and simple theories.
\newblock {\em Ann. Pure Appl. Logic}, 95(1-3):71--92, 1998.

\bibitem[Eri75]{ErimbetovACT}
M.~M. Erimbetov.
\newblock Almost categorical theories.
\newblock {\em Sibirsk. Mat. \v Z.}, 16:279--292, 420, 1975.

\bibitem[Gav06]{GavThesis}
Misha Gavrilovich.
\newblock {\em Model theory of the universal covering spaces of complex
  algebraic varieties}.
\newblock PhD thesis, Oxford University, 2006.

\bibitem[Gav09]{GavCoversAZ}
Misha Gavrilovich.
\newblock Covers of abelian varieties as analytic zariski structures.
\newblock 2009.
\newblock arXiv:0905.1377v1 [math.AG].

\bibitem[HH07]{HaHr07}
Assaf Hasson and Ehud Hrushovski.
\newblock D{MP} in strongly minimal sets.
\newblock {\em J. Symbolic Logic}, 72(3):1019--1030, 2007.

\bibitem[Hil12]{HilsGenAutGreen}
Martin Hils.
\newblock Generic automorphisms and green fields.
\newblock {\em J. Lond. Math. Soc.}, 85(1):223--244, 2012.

\bibitem[Hru92]{Hru92}
Ehud Hrushovski.
\newblock Strongly minimal expansions of algebraically closed fields.
\newblock {\em Israel J. Math.}, 79(2-3):129--151, 1992.

\bibitem[Kik00]{Kik00}
Hirotaka Kikyo.
\newblock Model companions of theories with an automorphism.
\newblock {\em J. Symbolic Logic}, 65(3):1215--1222, 2000.

\bibitem[KP06]{KowalskiPillayDGrps}
Piotr Kowalski and Anand Pillay.
\newblock Quantifier elimination for algebraic {$D$}-groups.
\newblock {\em Trans. Amer. Math. Soc.}, 358(1):167--181 (electronic), 2006.

\bibitem[Las85]{Las85}
Daniel Lascar.
\newblock Les groupes {$\omega$}-stables de rang fini.
\newblock {\em Trans. Amer. Math. Soc.}, 292(2):451--462, 1985.

\bibitem[Pil96]{Pil96}
Anand Pillay.
\newblock {\em Geometric stability theory}, volume~32 of {\em Oxford Logic
  Guides}.
\newblock The Clarendon Press Oxford University Press, New York, 1996.
\newblock Oxford Science Publications.

\bibitem[Poi83]{Poi83}
Bruno Poizat.
\newblock Une th\'eorie de {G}alois imaginaire.
\newblock {\em J. Symbolic Logic}, 48(4):1151--1170, 1983.

\bibitem[Poi01]{Poi01}
Bruno Poizat.
\newblock L'\'{e}galit\'{e} au cube.
\newblock {\em J. Symbolic Logic}, 66(4):1647--1676, 2001.

\bibitem[PP02]{PiPo02}
Anand Pillay and Wai~Yan Pong.
\newblock On {L}ascar rank and {M}orley rank of definable groups in
  differentially closed fields.
\newblock {\em J. Symbolic Logic}, 67(3):1189--1196, 2002.

\bibitem[PS03]{PillayScanlonMeroGrps}
Anand Pillay and Thomas Scanlon.
\newblock Meromorphic groups.
\newblock {\em Trans. Amer. Math. Soc.}, 355(10):3843--3859, 2003.

\bibitem[Rad04]{Rad04}
Dale Radin.
\newblock A definability result for compact complex spaces.
\newblock {\em J. Symbolic Logic}, 69(1):241--254, 2004.

\bibitem[Rit32]{RittFact}
J.~F. Ritt.
\newblock Erratum: ``{A} factorization theory for functions {$\sum\sp n\sb
  {i=1}a\sb ie\sp {\alpha\sb ix}$}'' [{T}rans.\ {A}mer.\ {M}ath.\ {S}oc. {\bf
  29} (1927), no. 3, 584--596; 1501406].
\newblock {\em Trans. Amer. Math. Soc.}, 34(4):938, 1932.

\bibitem[Roc]{RocheThesis}
Olivier Roche.
\newblock {\em Fusion d'un corps alg\'ebriquement clos avec un sous-groupe
  non-alg\'ebrique d'une vari\'et\'e ab\'elienne}.
\newblock PhD thesis, Universit\'e Lyon 1.
\newblock In preparation.

\bibitem[Sca06]{ScanlonNSMeroGrps}
Thomas Scanlon.
\newblock Nonstandard meromorphic groups.
\newblock In {\em Proceedings of the 12th {W}orkshop on {L}ogic, {L}anguage,
  {I}nformation and {C}omputation ({W}o{LLIC} 2005)}, volume 143 of {\em
  Electron. Notes Theor. Comput. Sci.}, pages 185--196 (electronic), Amsterdam,
  2006. Elsevier.

\bibitem[Zie06]{Zie04}
Martin Ziegler.
\newblock A note on generic types.
\newblock 2006.
\newblock arXiv:0608433v1 [math.LO].

\bibitem[Zil06]{ZCovers}
Boris Zilber.
\newblock Covers of the multiplicative group of an algebraically closed field
  of characteristic zero.
\newblock {\em J. London Math. Soc. (2)}, 74(1):41--58, 2006.

\bibitem[Zil10]{ZilberZariski}
Boris Zilber.
\newblock {\em Zariski geometries}, volume 360 of {\em London Mathematical
  Society Lecture Note Series}.
\newblock Cambridge University Press, Cambridge, 2010.
\newblock Geometry from the logician's point of view.

\end{thebibliography}

\end{document}